\documentclass[sn-mathphys,Numbered]{sn-jnl}
\usepackage{enumitem}
\usepackage{graphicx}%
\usepackage{multirow}%
\usepackage{amsmath,amssymb,amsfonts}%
\usepackage{amsthm}%
\usepackage{mathrsfs}%
\usepackage[title]{appendix}%
\usepackage{xcolor}%
\usepackage{textcomp}%
\usepackage{manyfoot}%
\usepackage{booktabs}%
\usepackage{algorithm}%
\usepackage{algorithmicx}%
\usepackage{algpseudocode}%
\usepackage{listings}%
\usepackage{bbm}

\theoremstyle{thmstyleone}%
\newtheorem{theorem}{Theorem}
\newtheorem{proposition}[theorem]{Proposition}%

\theoremstyle{thmstyletwo}%
\newtheorem{remark}{Remark}%

\theoremstyle{thmstylethree}%
\newtheorem{definition}{Definition}%

\begin{document}

\title[Article Title]{ A Zero-Sum Differential Game with Exit  Time}


\author*[1]{\fnm{Ekaterina} \sur{Kolpakova}}\email{eakolpakova@gmail.com}



\affil*[1]{
\orgname{Krasovskii Institute of Mathematics and Mechanics, Ural Branch of the Russian Academy of
		Sciences}, \orgaddress{\street{S. Kovalevskaya Str,16}, \city{Yekaterinburg}, \postcode{620108},  \country{Russia}}}




\abstract{The paper is concerned with a zero-sum differential game in the case where a payoff is determined by the exit time, that is, the first time when the system leaves the game domain. Additionally, we assume that a part of domain's boundary is a lifeline where the payoff is infinite.  

 Hereby,  the examined problem  generalizes   the well-known time-optimal problem as well as  time-optimal problem with lifeline. The main result of the paper relies on the solution to the Direchlet problem for the Hamilton-Jacobi equation associated with the game with exit time. We prove the existence of the value function for examined problem and construct  suboptimal feedback strategies under assumption that the associated Dirichlet problem for the Hamilton-Jacobi equation admits a viscosity/minimax solution.  Additionally, we derive a sufficient condition  of existence result to this  Dirichlet problem. }





\maketitle

\section{Introduction}\label{sec1}
The paper is concerned with a two-player differential game on a domain with a functional depending on an exit time and a realized trajectory. Namely, it is  assumed that we are given with 
\begin{itemize}
    \item a domain $G\in \mathbb{R}^d$;
    \item two sets: target set $M_1$ and lifeline $M_2$ such that $M_1\cap M_2=\varnothing$, while $M_1\cup M_2=\partial G$;
    \item a control system governed by two players
    \[\frac{d}{dt}x(t)=f(x(t),p(t),q(t)).\]
\end{itemize}
We let $\tau(x(\cdot))$ be an exit time from $G$, i.e., $\tau(x(\cdot))$ is the first time when the motion $x(\cdot)$ touches the set  $\partial G$. The quality of the realized control process is evaluated as \[\sigma(x(t))+\int_0^{\tau(x(\cdot)) }g(x(s))ds.\] We  assume that the terminal part  of payoff $\sigma$  takes bounded values on $M_1$ and is equal to $+\infty$ on $M_2$, whilst $g$ is strictly positive.

The examined problem generalizes a time optimal differential game as well as a time-optimal problem with a lifeline. Indeed, if we put $M_2=\varnothing$, $\sigma\equiv 0$ and $g\equiv 1$, we obtain the standard time-optimal differential game. Letting $M_2$ to be nonempty
\[\sigma(x)\triangleq\left\{\begin{array}{cc}
     0, & x \in M_1 \\
     +\infty, & x \in M_2 
\end{array}\right., \ \ g\equiv 1,\] we arrive at the time-optimal differential game with lifeline.

This  differential game with exit time was first formulated in~\cite{Krasovskii}. Surprisingly, up to now it was not examined. The case of general payoff depending on exist time was studied only for the case of one decision maker in~\cite{Ye,Cannarsa}. Simultaneously, there is a great interest to time-optimal differential games and  time optimal differential games with lifeline.
Apparently, the study of these problems was initiated by the seminal book by R.~Isaacs~\cite{Isaaks}.
Nowadays, there are three approaches in this area of differential games. 

The first approach uses various finite-dimensional geometric constructions to design suboptimal feedback strategies (see~\cite{Pachter,Pachter1, Petrosjan, Petrosjan1, Basar,yan}) and references therein. The authors consider given geometrical form of a target set or a lifeline.

The second approach relies on viscosity/minimax solution to the Dirichlet problem for the Hamilton-Jacobi equation associated with the differential game. Within this approach it is shown that the value function of the time-optimal game exists and  is determined by the solution of the Dirichlet problem~\cite{Bardi, Subbotin}. Additionally, it is proved that given a viscosity/minimax solution to the Dirichlet problem of Hamilton-Jacobi equation associated with differential game, one can construct optimal strategies of the players (see~\cite{Subbotin} for standard time-optimal differential game and~\cite{Kumkov} for time optimal differential game with lifeline). The papers dealing with optimal control problem with exist time~\cite{Ye,Cannarsa, Motta} also follows this framework.

The third approach is close to the previous one. Within it, tools of the viability theory are involved  to derive domains where the first or second player wins~\cite{Cardaliaguet,Cardaliaguet1,Patsko,Krasovskii}. The players' optimal strategies are constructed using these domains.  Recall also that this approach leads to the description of value function due to the fact that it can be characterized in the terms of viability theory~\cite{Krasovskii, Subbotin}.  Additionaly, the authors assume the dynamic advantage of one or both players in the neighbourhood of the boundary of a target set or a lifeline.

The approach of the paper  is close to the second approach. We adopt the feedback formalization by first proposed by Krasovskii and Subbotin in~\cite{Krasovskii}. Recall that this formalization implies that the players form their controls stepwise while
the strategy is an arbitrary function of the state. The main result of the paper is the construction of  suboptimal players' strategies. It relies on a viscosity/minimax solution of the Dirichlet problem for Hamilton---Jacobi equation that is associated with examined differential game with exit time. The general scheme follows the approach proposed in~\cite{Subbotin} for standard time optimal differential game. Moreover, we present a condition that guarantees the existence of the viscosity/minimax solution in this Dirichlet problem.

The rest of the paper is organized as follows. In Sect.2 we introduce the differential game with lifeline. The feedback formalization and the definition of the value function are discussed in Sect.3. The next section provides a condition guaranteeing the existence of the viscosity/minimax solution to the Dirichlet problem associated with the differential game. Finally, the main result is in Sect.5. Here, we construct the players' suboptimal strategies and proof of the existence of the value function for the differential game with exit time. The latter is equal to the Kruzhkov transform of the viscosity/minimax solution to Dirichlet  problem for Hamilton---Jacobi equation.

\section{Problem setting}\label{sec2}

We study a differential game with the dynamics given by
\begin{equation} \label{Kolpakova:eq2}
	\dot{x}(t)=f(x(t),p(t), q(t)), \ 
\end{equation}
and the initial condition $x(0)=x_0.$
Here, $t\geq0$; $x(t)\in \mathbb{R}^d$ is the state of the system at the time $t$; the control of the first player is $p(t)\in P$; the control of the second player is $q(t)\in Q$; $P,Q$ are compacts.
To introduce the payoff, we consider two  closed sets  $M_1\subset \mathbb{R}^d$, $M_2\subset \mathbb{R}^d$ such that  the   distance between them is strictly positive. The boundary of the sets $M_1$ and $M_2$ is continuous. The area of the game is  $G=\mathbb{R}^d\setminus (M_1\cup M_2)$, $\partial G$ is the boundary of the set $M_1\cup M_2$. We denote by the symbol $\mathcal{M}_1$ the boundary of the set $M_1$ and by the symbol $\mathcal{M}_2$ the  boundary of the set $M_2$. 
 
 We define the  functional on the space of continuous functions $x(\cdot): \mathbb{R}^+ \rightarrow \mathbb{R}^d$ by the rule
\begin{equation}\label{time}
	\tau(x(\cdot))=\min\{ t\geq0: x(t)\in M_1\cup M_2\}.
\end{equation}
If $x(t)\not\in M_1\cup M_2$ for all $t\geq0$, then we set $\tau(x(\cdot))=+\infty$. The quantity $\tau(x(\cdot))$ is the exit time for the trajectory $x(\cdot)$.

For $(x(\cdot),p(\cdot),q(\cdot))$ satisfying~\eqref{Kolpakova:eq2}, the payoff functional is equal to
\begin{equation}\label{funcl}
	J(x(\cdot),p(\cdot),q(\cdot))=\int\limits_{0}^{\tau(x(\cdot))} g(x(t))dt+ \sigma(x(\tau(x(\cdot)))),
\end{equation}
where  $\tau(x(\cdot))$ is defined by~\eqref{time}. We assume that the first (respectively second) player tries to minimize (respectively maximize) the payoff functional.
Notice, that if $\tau(x(\cdot))=+\infty$ then $J(x(\cdot),p(\cdot),q(\cdot))=+\infty$.

We impose the following assumptions:
\begin{description}
	\item[A1] the functions $f$, $g$  are continuous and the function $f$ satisfies the sublinear condition:
	$$ \|f(x,p,q)\|\leq R_f(1+\|x\|)  \ \forall  (x,p,q)\in\mathbb{R}^d \times P\times Q ;$$ 	
	 \item[A2] the functions $f$, $g$ satisfy the Lipschitz condition w.r.t $x$:
	$$||f(x+y,p,q)-f(x,p,q)|| +||g(x+y)-g(x)||\leq \lambda\|y\|,$$  for each $  x,y\in \mathbb{R}^d$,  $p\in P$, $q\in Q$;
	\item[A3] the saddle-point condition in a small game (the Isaacs condition) is fulfilled:  for every  $s\in \mathbb{R}^d$, $x\in \mathbb{R}^d$  $$\min\limits_{p\in P}\max\limits_{q\in Q} \left\langle s, f(x,p,q) \right\rangle=\max\limits_{q\in Q} \min\limits_{p\in P} \left\langle s, f(x,p,q) \right\rangle; $$
	\item[A4] there is a constant $b>0$ such that $\forall x\in \mathbb{R}^d$ $b\leq g(x)$;
	\item[A5] there exists a constant $\Sigma>0$ such that $\sigma(x)$ takes values in $[0, \Sigma]$ wherever $ x\in \partial M_1$ and $\sigma(x)=+\infty$, if $ x\in \partial M_2$. Additionally, 
  the function $\sigma$  is Lipschitz continuous for some constant $L$ on the set $\partial M_1$.
  \end{description}

\begin{remark}
If the function $g(x)\equiv 1$ and $\sigma(x)\equiv 0$, as $\in \mathcal{M}_1$, the set $M_2$ is empty,
then	problem~\eqref{Kolpakova:eq2}--\eqref{funcl} is reduced to the time optimal problem.

If the function $g(x)\equiv 1$; $\sigma(x)\equiv 0,$ as $ x\in \mathcal{M}_1$  and $\sigma(x)=+\infty,$ as $ x\in \mathcal{M}_2$, then	problem~\eqref{Kolpakova:eq2}--\eqref{funcl} is reduced to the time optimal problem with lifeline.
	\end{remark}

\section{Feedback strategies and value function}\label{sec3}
At first, let us describe feedback strategies of the players within the framework of Krasovskii---Subbotin approach~\cite{Krasovskii}.

A function $U:G \rightarrow P$ (respectively, a function $V:G \rightarrow Q$) is  called a
feedback strategy of the first (respectively, the second) player.
Now, let us introduce the motions generated by the strategies. We start with the first player.
Let a strategy $U$ and an  initial point $x_0 \in \mathbb{R}^d$ be given.
Further, we assume that the first player  chooses a countable set of time corrections
$\Delta=\{t_i\}_{i=0}^\infty$ such that $t_{i+1}>t_i$ and $\inf(t_{i+1}-t_i)>0$.

A corresponding step-by-step motion $x(\cdot)$ on each interval $[t_i,t_{i+1})$ satisfies the following conditions:
\begin{equation} \label{step1eq}
  \dot{x}(t)=f(x(t),U(x(t_i)),q(t)), \ \forall q(\cdot)\in L^\infty([t_i,t_{i+1}],Q),   t\in[t_i,t_{i+1}),   i=0,1,2,\ldots  
\end{equation}
 with initial condition $x(0)=x_0$, that is the set of step-by-step
motions of the differential equation~\eqref{step1eq}, generated by $U$.
We denote by the symbol $X(x_0,U,\Delta)$  this set of step-by-step motions.

Similarly, the second player chooses a feedback strategy $V:\mathbb{R}^d \rightarrow Q$ and a partition $\Delta$.   We construct the set of continuous functions  $x(\cdot)$  on each interval  $[t_i,t_{i+1})$ satisfying the differential equation
\begin{equation} \label{str-dif}
   \dot{x}(t)= f(x(t),p(t),V(x(t_i))),  \  \forall p(\cdot)\in L^\infty([t_i,t_{i+1}],P), t\in[t_i,t_{i+1}),   i=0,1,2,\ldots 
\end{equation}
with initial condition $x(0)=x_0$, that is the set of step-by-step
motions of the differential equation~\eqref{str-dif}, generated by $V$.
The symbol $X(x_0,V,\Delta)$ denotes this set of step-by-step
motions.

Due to stepwise constructions we are swell the target set.
In the following, we denote 
$$M_i^\varepsilon=\{x+y:x\in \mathcal{M}_i, \|y\|\leq \varepsilon\}, \ i=1,2.$$

Let
\begin{equation}\label{time-e}
	\tau_\varepsilon(x(\cdot))=\min\{t\geq 0:x(t)\in M_1^\varepsilon \cup M_2^\varepsilon\}.
\end{equation}
If $x(t)\not\in M_1^\varepsilon\cup M_2^\varepsilon$ for all $t\geq 0$, then we put $\tau_\varepsilon(x(\cdot))=+\infty$.

We extend the function $\sigma$ defined on $\mathcal{M}_1\cup\mathcal{M}_2$ to the set $G$ such that the extended function $\hat{\sigma}$ satisfies the following conditions: 
\begin{itemize}
    \item $\hat{\sigma}(x) |_{\mathcal{M}_1}= \sigma(x)$,
    \item $\hat{\sigma}:G\setminus M_2^\varepsilon \rightarrow [0,\Sigma)$ is Lipschitz continuous,
    \item $\hat{\sigma}=+\infty,$ as $x\in M_2^\varepsilon$.   
\end{itemize}     

\begin{remark}
    Notice, that the function defined by the rule$$
\hat{\sigma}(x)=\max[ \sup\{ \sigma(y)-L\|x-y\|: y\in \mathcal{M}_1\},0] , x\in G\setminus M_2^\varepsilon.
$$ meets aforementioned conditions. \end{remark} 
\begin{proof}
    Indeed,
since $\mathcal{M}_1$ is a closed set and the function $\sigma$ is bounded, we have that the function $\hat{\sigma}$ is well defined. So, we can apply McShane---Whitney theorem for extension of $\sigma$~\cite{Petrakis}. According to this theorem, the function $\hat{\sigma}$  is Lipschitz continuous on the set $G\setminus M_2^\varepsilon$ and $\hat{\sigma}|_{\mathcal{M}_1}=\sigma$.
\end{proof}

Let us introduce the following functionals.
$$ J_1^\varepsilon(x_0,U,\Delta)=\sup\Biggl\{\int\limits_{0}^{\tau_\varepsilon(x(\cdot))} g(x(t))dt+ \hat{\sigma}(x(\tau_\varepsilon(x(\cdot)))):x(\cdot)\in X(x_0,U,\Delta)\Biggr\},$$
$$J_1^\varepsilon(x_0,U)=\limsup\limits_{\mbox{diam}(\Delta)\rightarrow 0} J_1^\varepsilon(x_0,U,\Delta),$$
$$J_1^\varepsilon(x_0)=\inf\limits_{U}J_1^\varepsilon(x_0,U).$$ Finally, let
$$J_1^0(x_0)=\limsup\limits_{\varepsilon\rightarrow 0 }J_1^\varepsilon(x_0)\in[0,+\infty].$$
The quantity $J_1^0(x_0)$ is the guarantee of the first player in the feedback strategies.

Similarly, we define the  guarantee for the second player.
$$ J_2^\varepsilon(x_0,V,\Delta)=\inf\Biggl\{\int\limits_{0}^{\tau_\varepsilon(x(\cdot))} g(x(t))dt+ \hat{\sigma}(x(\tau_\varepsilon(x(\cdot)))):x(\cdot)\in X(x_0,V,\Delta)\Biggr\},$$
$$J_2^\varepsilon(x_0,V)=\liminf\limits_{\mbox{diam}(\Delta)\rightarrow 0} J_2^\varepsilon(x_0,V,\Delta),$$
$$J_2^\varepsilon(x_0)=\sup\limits_{V}J_2^\varepsilon(x_0,V),$$
$$J_2^0(x_0)=\liminf\limits_{\varepsilon\rightarrow 0 }J_2^\varepsilon(x_0)\in[0,+\infty].$$

For arbitrary control processes   $(U,\Delta^1)$ and $(V, \Delta^2)$, the following inequalities are valid~\cite{Subbotin}:
$$J_2^\varepsilon(x_0,V,\Delta^2)\leq J_1^\varepsilon(x_0,U,\Delta^1), \quad J_2^0(x_0)\leq J_1^0(x_0).$$

If $J_2^0(x_0)= J_1^0(x_0)$, then we say that the value $\operatorname{Val}$ exists in differential game and
$\operatorname{Val}(x_0)=J_2^0(x_0)= J_1^0(x_0)$.

\section{Minimax solution for the Hamilton---Jacobi equation}\label{sec4}
In this section  for the description of the value function, we will use the notion of the generalized solution for the Hamilton---Jacobi equation. 
Let us
define the Hamiltonian $$ \mathcal{H}(x,s)=\Bigl[\min\limits_{p\in P}\max\limits_{q\in Q}\left\langle s,f(x,p,q)\right\rangle+g(x)\Bigr].$$

We consider the Dirichlet problem for the Hamilton---Jacobi equation
\begin{equation} \label{h1}
	\mathcal{H}(x,\nabla\varphi(x))=0, x\in G; \ \varphi(x)=\sigma(x),  \ x\in \partial G.
\end{equation}
Here, the function $\varphi$ can take infinity values. We apply the Kruzhkov's transform to function $\varphi(\cdot)$:
$$ u(x)=1-e^{-\varphi(x)}.$$
Since $\varphi$  takes values in $[0,+\infty]$, we obtain that
 $u(x)\in[0,1]$ $\forall x\in G$. We denote by the symbol $\tilde{\sigma}$ the Kruzhkov's transformation of function $\hat{\sigma}:$
$$\tilde{\sigma}(x)=1-e^{-\hat{\sigma}(x)}. $$ To write down equation on $u$, we put
 $$H(x,s,z)=\min\limits_{p\in P}\max\limits_{q\in Q}\left\langle s,f(x,p,q)\right\rangle+g(x)(1-z).$$
Notice, that the function $\varphi$ satisfies ~\eqref{h1} iff the corresponding function $u$ satisfies the Dirichlet problem for the Hamilton---Jacobi equation
\begin{equation} \label{h2}
	H(x,\nabla u(x),u(x))=	0, x\in G; u(x)=\tilde{\sigma}(x), \ x\in \partial G.
\end{equation}

Let us divide the  Hamilton---Jacobi equation~\eqref{h2}  by the coefficient $g(x)$. Recall that $g(x)>0$.Thus, we arrive to the problem:
\begin{equation}\label{h3}
	\min\limits_{p\in P}\max\limits_{q\in Q}\left\langle D_xu(x),f(x,p,q)\right\rangle+1-u(x)=0,  u(x)=\tilde{\sigma}(x), \ x\in G.
\end{equation}

A solution of problem~\eqref{h3} is understood  in the minimax sense. It is equivalent to the notion of the viscosity solution~\cite{Subbotin}. According to Subbotin's approach let us recall the definition of the generalized (minimax) solution of problem~\eqref{h3} and the definition of upper and lower directional derivatives at the point $x$ in the direction $f$~\cite{Subbotin}: $$d^+u(x;f)=\lim\limits_{\varepsilon\rightarrow0} \Bigl\{\sup \frac{u(x+\delta f')-u(x)}{\delta}: \ (\delta,f')\in B_\varepsilon(x,f)\Bigr\}; $$
$$d^-u(x;f)=\lim\limits_{\varepsilon\rightarrow0} \Bigl\{\inf\frac{u(x+\delta f')-u(x)}{\delta}: \ (\delta,f')\in B_\varepsilon(x,f) \Bigr\}. $$
Here, $B_\varepsilon(x,f)=\{(\delta,f')\in (0,\varepsilon)\times \mathbb{R}^d: \ \|f-f'\|\leq \varepsilon, x+\delta f' \in G \}$.
\begin{definition}
	A supersolution  of problem~\eqref{h2} is a lower semicontinuous  function
	$u:\mbox{cl } G\rightarrow \mathbb{R}$ such that
	\begin{enumerate}
		\item  $u(x)=\tilde{\sigma}(x), \ x\in \partial G$, for some constant $c>0$ $\sup\limits_{x\in \mbox{cl }G}|u(x)|\leq c$;
	\item $\inf\{d^-u(x;\bar{f})-\bar{g}: (\bar{f},\bar{g})\in E^+(x,u(x),q)\}\leq 0$, for each
  $ x\in G$, $q\in Q$. Here
  		\begin{equation}\label{d-incl}
			E^+(x,u(x),q)=\mbox{co }\{ (f(x,p,q),\bar{g})\in \mathbb{R}^d \times \mathbb{R}: p\in P, \bar{g}=g(x)(z-1)\},
					\end{equation}
		\end{enumerate}
\end{definition}

\begin{definition}
	A subsolution of problem~\eqref{h2} is an upper semicontinuous function
	$u:\mbox{cl } G\rightarrow \mathbb{R}$ such that
	\begin{enumerate}
		\item  $u(x)=\tilde{\sigma}(x), \ x\in \partial G$, for some constant $c>0$ $\sup\limits_{x\in \mbox{cl }G}|u(x)|\leq c$;
		\item $u(x)$ is continuous at each point $ x\in \partial G$;
		\item $\sup\{d^+u(x;\bar{f})-\bar{g}: (\bar{f},\bar{g})\in E^-(x,u(x),p)\}\geq 0$, for each
  $ x\in G$, $ p\in P$. Here
   		   \begin{equation}\label{d-incl1}
		   		E^-(x,u(x),p)=\mbox{co }\{ (f(x,p,q),\bar{g})\in \mathbb{R}^d \times \mathbb{R}: q\in Q, \bar{g}=g(x)(z-1)\},
				\end{equation}   
	\end{enumerate}
\end{definition}

\begin{definition}
	A minimax solution of problem~\eqref{h2} is  a function $u:\mbox{cl } G\rightarrow \mathbb{R}$, satisfying the equality at each $x\in \mbox{cl } G$
	$$\lim_{k\rightarrow \infty}u^k(x)=u(x)=\lim_{k\rightarrow \infty}u_k(x) ,$$
	where $\{u^k\}_{k=1}^\infty$ $(\mbox{respectively, }\{u_k\}_{k=1}^\infty)$ is a sequence of  supersolutions (respectively, subsolutions)  of problem~\eqref{h2}.
\end{definition}

\begin{theorem}\cite{Subbotin}
Assume that there exists a subsolution of  problem
\eqref{h3}. Then, there exists a unique minimax solution of  problem~\eqref{h3}. The minimax solution coincides with the minimal supersolution.
\end{theorem}

Thus,  the existence result  for problem~\eqref{h2} is reduced to the   existence of the subsolution  to problem~\eqref{h3}.
Further, we construct the proper subsolution for problem~\eqref{h2}. 

Let us consider the function \begin{equation} \label{subsol}
v(x)=\begin{cases}
1, \ x \in M^\varepsilon_2,\\
\tilde{\sigma}(x), \ x\in G \setminus M^\varepsilon_2.
\end{cases}
\end{equation}

We introduce the following condition  for the function $v$ defined by~\eqref{subsol}:
  \begin{equation} \label{cond-low} \forall x_0 \in  G \ 
 \forall \ p\in P  \ \exists \  q_0\in Q : \ b(1-v(x_0))\geq-d^+v(x_0;f(x_0,p,q_0)),  
 \end{equation} 
\begin{proposition}  Assume that  condition~\eqref{cond-low} is true. Then
   the function $v$ of form~\eqref{subsol} is a subsolution of  Dirichlet problem~\eqref{h3}. 
\end{proposition}
\begin{proof}
    The function $v$ is upper semicontinuous in $G$, continuous at each point of $ x\in \partial G$ and it is bounded due to the definition.
    Let us check the inequality for each $x_0 \in  G,
 \ p\in P $:
     $$\sup\{ d^+v(x_0;\bar{f})-\bar{g}:( \bar{f},\bar{g})\in E^-(x,v(x_0),p)\}\geq 0.$$
     \begin{itemize}
         \item  If $x_0\in G\setminus M_2^\varepsilon$, then we choose $\bar{f}$ from condition~\eqref{cond-low} and $\bar{g}=g(x_0)(v(x_0)-1)$. Hence, $$ d^+v(x_0;\bar{f})-\bar{g}\geq d^+v(x_0;\bar{f})+b(1-v(x_0))\geq 0.$$ 
         \item If $x_0$ is the inner point of the set $M_2^\varepsilon$, then $$v(x)\equiv1, d^+v(x_0;\bar{f})=0,  \bar{g}=g(x_0)(v(x_0)-1)=0. $$ So, $d^+v(x_0;\bar{f})-\bar{g}=0.$
         \item  If $x_0\in \partial M_2^\varepsilon$ and $v(x_0)=1$,
   then from~\eqref{cond-low}, we choose $\bar{f}$. 
   \subitem  If $x_0+\delta \bar{f} \in   M_2^\varepsilon$, then $d^+v(x_0;\bar{f})=0$, $\bar{g}=g(x_0)(v(x_0)-1)=0$. 
   \subitem If $x_0+\delta \bar{f} \in   G\setminus M_2^\varepsilon$, then  $d^+v(x_0;\bar{f})=-\infty$ and it contradicts to~\eqref{cond-low}.  Thus,
   $d^+v(x_0;\bar{f}) -\bar{g}=0$ and 
 $$\sup\{ d^+v(x;\bar{f})-\bar{g}:( \bar{f},\bar{g})\in E^-(x,v(x),p)\}= 0, \ \forall \ x\in M_2^\varepsilon, p\in P.$$
     \end{itemize}
   
    \end{proof}
    
Condition~\eqref{cond-low} provides the existence of a subsolution of Hamilton---Jacobi equation~\eqref{h3} and low semicontinuity of the value function.

   \begin{remark}
Condition ~\eqref{cond-low} always holds in the time optimal problem.\end{remark}
\begin{proof}
    Indeed, for the time optimal problem $g(x)\equiv 1$, $\varsigma(x)\equiv0$, for $x\in M_1$ and $M_2$ is empty.  Notice, that $\tilde{\sigma}(x)\equiv 0$,  then $dv(x; f(x,p,q))=d \tilde{\sigma}(x;f(x,p,q))\equiv0$, for any $p\in P$, $q\in Q$ and $x\in G$. So, we have $1-v(x)=1>0=-dv(x; f(x,p,q))$.
\end{proof}

Example. We change a little  an example from~\cite{Subbotin}.
Let us consider the dynamic system with the functional:
$$\dot{x}(t)=1, \quad J(x_0)=\int\limits_0^t ds, x\in[0,1].$$ The target set $M_1=\{0\}$, the lifeline $M_2=\{1\}$.
We introduce the Hamiltonian 
$$H(x,s,z)=s+1-z.$$ Then the Dirichlet problem has the form:
$$\frac{du}{dx}+1-u(x)=0, \ \tilde{\sigma}(0)=0, \ \tilde{\sigma}(1)=1, 0<x<1.$$

The subsolution is $$v_\epsilon(x)=\begin{cases}
1, \ x \in [1-\epsilon,1],\\
0, \ [0,1-\epsilon),
\end{cases}$$
$\epsilon\in (0,1)$.
We show that $v_\epsilon(\cdot)$ is a subsolution of the Dirichlet problem.
Notice, that the function $v$ is upper semicontinuous, it is bounded and continuous at the points $x=0$ and $x=1$. 
 Further, if $x_0\in [0,1-\epsilon)$, then $ \bar{f}=1$, $\bar{g}=-1$, $d^+v_\epsilon(x_0;1)=0<b=1$. So,
 $$\sup\{ d^+v_\epsilon(x_0;1)-\bar{g}\}=1\geq 0. $$
 If $x_0\in (1-\epsilon,1]$, then $ \bar{f}=1$, $\bar{g}=0$,  $d^+v_\epsilon(x_0;1)=0$,
  and $\sup\{ d^+v_\epsilon(x_0;1)-\bar{g}\}=0\geq 0.$
 If $x_0=1-\varepsilon$, then $ \bar{f}=1$, $\bar{g}=0$, $d^+v_\epsilon(x_0;1)=0$, and $\sup\{ d^+v_\epsilon(x_0;1)-\bar{g}\}= 0.$ Condition~\eqref{cond-low} is fulfilled in this problem.
 
 Hence, there exists the minimax solution $u(\cdot)$ in this problem~\cite{Subbotin}.
The minimax solution has the form due to the direct calculations $$u(x)=\begin{cases}
1, \ x \in (0,1],\\
0, \ x=0.
\end{cases}$$
 It is low semicontinuous function, coincides with a supersolution of the Dirichlet problem, satisfies the Hamilton---Jacobi equation and the boundary condition.

\section{Construction of $\varepsilon$-optimality strategies}\label{sec5}
In this section, we assume that we are given with a unique solution of the Dirichlet problem  to the Hamilton---Jacobi equation. We aim to construct the players suboptimal strategies based on this function. 

Let $u:\mathbb{R}^d\rightarrow [0,1]$ be a unique minimax solution of problem~\eqref{h2}.
Let us transform the function $u$
\begin{equation}\label{u-a}
	u_\alpha(x)=\min_{y\in \mathbb{R}^d}[ u(y)+w_\alpha(x,y)],
\end{equation}
where
\begin{equation}\label{w}
	w_\alpha(x,y)=\frac{(\alpha^{\frac{2}{\nu}} +\|x-y\|^2)^\nu}{\alpha}, \ \nu=\frac{1}{2+2\lambda}
\end{equation}
Here, $\lambda$ is the Lipschitz constant from assumption $A2$.

The function $y\rightarrow u(y)+w_\alpha(x,y)$ is lower semicontinuous,  therefore  minimum in expression~\eqref{u-a} is attained at the point $y_\alpha$  such that $\|x-y_\alpha\|\leq1$. It follows from the work~\cite{Subbotin} that $\|x-y_\alpha\|\leq 2\alpha$.

\begin{proposition}  There exists  $\alpha_0>0$ such that for all $\alpha \in (0,\alpha_0]$ the inequality
	\begin{equation} \label{ineq}
	H(x,D_xw_\alpha(x,y),z)-H(y,-D_yw_\alpha(x,y),z)-bw_\alpha(x,y)\leq0
	\end{equation}
	is valid for all $x,y\in \mathbb{R}^d$, $\|x-y\|\leq 1$, $z\in \mathbb{R}$, the function $w_\alpha$ is defined by~\eqref{w}. The constant $b$ is defined in assumption $A4$.
\end{proposition}
\begin{proof}
	Let us set $s=D_xw_\alpha(x,y)=-D_yw_\alpha(x,y)$
	and estimate the expression
	\begin{equation*} \begin{split}	
		& bw_\alpha(x,y)-H(x, D_xw_\alpha(x,y),z)+H(y,-D_yw_\alpha(x,y),z)\\
	& =bw_\alpha(x,y)-H(x, s,z)+H(y,s,z)\\
	& \geq bw_\alpha(x,y) -\lambda||x-y||\geq b\frac{\|x-y\|^{2\nu}}{\alpha}-\lambda||x-y||.
\end{split}
	\end{equation*}
	For $\|x-y\|\leq 1$, we obtain that $\|x-y\|^{2\nu}\geq \|x-y\|$.
	Then, $$ b\frac{\|x-y\|^{2\nu}}{\alpha}-\lambda||x-y||\geq \Bigl(\frac{b-\lambda \alpha}{\alpha}\Bigr)\|x-y\|\geq 0.$$
	Finally, let us choose $\alpha_0$  such that $$ b\geq\lambda\alpha_0 .$$
\end{proof}

Now, we define the pre-strategies, realizing the extremal shift rule:
\begin{equation}\label{predstr}
	p_0(x,s)\in \arg\min_{p\in P}\Bigl\{ \max\limits_{q\in Q} \langle s,f(x,p,q)\rangle \Bigr \},
	q_0(x,s) \in \arg\max_{q\in Q}\Bigl\{ \min\limits_{p\in P} \langle s,f(x,p,q)\rangle \Bigr\}.
\end{equation}

The feedback strategy
 $U_\alpha:\mathbb{R}^d\rightarrow P$ is defined by
\begin{equation}\label{str-u}
	U_\alpha(x)=p_0(x,s_\alpha(x)),
\end{equation}
where $p_0$ satisfies~\eqref{predstr}, the vector $s_\alpha(x)=D_xw_\alpha(x,y_\alpha(x))=-D_yw_\alpha(x,y_\alpha(x)),$
$y_\alpha(x) \in \arg\min\limits_{y\in \mathbb{R}^d} [u(y)+w_\alpha(x,y)]$.

\begin{theorem} Let $u$ be a minimax solution of problem~\eqref{h2} and $x_0 \in G$ be such that $u(x_0)<1$.  Then, for every $\varepsilon>0$ and $I>-\ln(1-u(x_0))$,
	there exist $\alpha>0$ and $\delta_0>0$ satisfying:
	$$ J_1^\varepsilon(x_0,U_\alpha,\Delta)\leq I   \mbox{ while diam }\Delta <\delta_0.$$
		Here, $U_\alpha$ is a  strategy satisfying~\eqref{str-u}.
\end{theorem}
\begin{proof}
Let us denote  the set of solutions of the differential inclusion
	$$\dot{x}(t)=\mbox{co } \{ f(x(t),p,q): p\in P, q\in Q\}, \quad x(0)=x_0$$
by  $X(x_0)$.
	We choose arbitrary $\theta>I/b$.
	Let us define the set \begin{equation} \label{setk} 
	 K=\{x(t)\in \mathbb{R}^d: x(\cdot)\in X(x_0), t\in[0,\theta]\}.   
	\end{equation}
	Additionally, we denote 
 \begin{equation} \label{kapa}  \kappa=\max_{x\in K}g(x), \ m=\sup\{\|f(x+h,p,q)\|, x\in K, p\in P,q\in Q, \|h\|\leq1\}.\end{equation}
	Notice, that the set $K$ is bounded and $m<\infty$. Further, let us  choose   $\alpha$ and $\delta_0$ such that
	\begin{equation} \label{l1}
	 3\alpha\leq\varepsilon, \ \delta_0m\leq \alpha, \ 3\alpha<1.   
	\end{equation}

	We choose a partition $\Delta$ and consider a step-by-step motion $x(\cdot) \in X(x_0,U_\alpha, \Delta)$.  Further, we suppose that $t_i\in \Delta$, $t_i<\theta$ and $\mbox{dist}(x(t_i);\partial G)>\varepsilon$. Then, we will prove that for all $r\in [t_i,t_{i+1}]\cap [0,\theta]$
	\begin{equation} \label{main}
		1-u_\alpha(x(t_i))\leq e^{-\int_{t_i}^r g(x(t))dt}(1-u_\alpha(x(r)))+h(\delta_0,x_0,\theta)(r-t_i),
	\end{equation}
	where $\lim\limits_{\delta_0\rightarrow 0} h(\delta_0,x_0,\theta)=0$. The function $h(\delta_0,x_0,\theta)$ depends
only on $\delta_0,x_0,\theta$ and does not depend on $x(\cdot) \in X(x_0,U_\alpha, \Delta)$.
		Additionally, we  assume that
	\begin{equation}\label{pars}
	    \alpha+\theta h(\delta_0,x_0,\theta)<\Bigl[e^{b\theta+\ln(1-u(x_0))}-1\Bigr]e^{-\kappa\theta},
	\end{equation}
	the constant $b$ is defined by $A4$. From the definition of $\theta$, the quantity $$e^{b\theta+\ln(1-u(x_0))}-1>0.$$ Hence, the quantities $\alpha$ and $\delta_0$ are well defined.
	
	Now,  we prove inequality~\eqref{main}.
	Let us denote
	$$\xi=x(t_i), \eta=y_\alpha(\xi),s_*=s_\alpha(\xi), p_*=U_\alpha(\xi)=p_0(\xi,s_*),q_*=q_0(\eta,s_*).$$
	Since the functions $f,g$ and feedback strategy $U_\alpha$ do not depend on time, one can consider the case when $t_i=0$.
	Let us put $f^*=\frac{1}{r}\int\limits_{0}^r \dot{x}(t)dt$,
	where $$\dot{x}(t)\in \mbox{co }\{f(x(t),p_*,q):q\in Q\}, \quad x(0)=\xi.$$
	Thus, to prove~\eqref{main}, it sufficient to show that
 $$1-u_\alpha(\xi)\leq e^{-\int_0^r g(x(t))dt} (1-u(\xi+f^*r))+ r h(r,x_0,\theta).$$
	It  follows from the inequalities $ \mbox{dist}(\xi;\partial G)>\varepsilon$, $3\alpha<\varepsilon$ and $\|\xi-\eta\|\leq 2\alpha$  that
	$\mbox{dist}(\eta;\partial G)>\alpha$. From  this and the inequality $\delta_0m\leq \alpha$, we conclude that $y(t)\not\in \partial G$ for every solution of differential inclusion~\eqref{d-incl}  and for all $t\in[0,r]$.

	The viability of epi $u$ w.r.t. differential inclusion~\eqref{d-incl} implies
	\begin{equation} \label{inv2}
		1-u(\eta)\leq e^{-\int_0^r g(y(t))dt} (1-u(\eta+f_*r)).
	\end{equation}
	Here, $f_*=\frac{1}{r}\int\limits_{0}^r \dot{y}(t)dt$,
	while $\dot{y}(t)\in \mbox{co} \{f(y(t),p,q_*):p\in P\}$, for all $t\in [0,r]$, $y(0)=\eta$.
	According to  definition~\eqref{u-a}, we have
	$$u_\alpha(\xi)=u(\eta)+w_\alpha(\xi,\eta). $$
	This equality and~\eqref{inv2}  provide
	$$ 1-u_\alpha(\xi)\leq e^{-\int_0^r g(y(t))dt} (1-u(\eta+f_*r))-w_\alpha(\xi,\eta).$$
	
	Adding to and subtracting from the right-hand side of the last inequality
	the quantity $e^{-\int_0^r g(y(t))dt}w_\alpha(\xi+f^*r,\eta+f_*r)$, we obtain
	\begin{equation*} \begin{split}
	&	1-u_\alpha(\xi)\leq e^{-\int_0^r g(y(t))dt} \Bigl(1-u(\eta+f_*r)-w_\alpha(\xi+f^*r,\eta+f_*r)\Bigr)\\
	& -w_\alpha(\xi,\eta)+e^{-\int_0^r g(y(t))dt}w_\alpha(\xi+f^*r,\eta+f_*r).
		\end{split}
	\end{equation*}
	From the definition $u_\alpha$, we obtain that $$ u_\alpha(\xi+f^*r)\leq u(\eta+f_*r)+w_\alpha(\xi+f^*r,\eta+f_*r)). $$
	Thus,\begin{equation}\label{ineq2} \begin{split}
	&	1-u_\alpha(\xi) \leq e^{-\int_0^r g(y(t))dt} (1-u_\alpha(\xi+f^*r))\\
	&	+e^{-\int_0^r g(y(t))dt}w_\alpha(\xi+f^*r,\eta+f_*r)-w_\alpha(\xi,\eta).
		\end{split}
	\end{equation}
	
	Let us denote $\gamma=e^{-\int_0^r g(y(t))dt}w_\alpha(\xi+f^*r,\eta+f_*r)-w_\alpha(\xi,\eta)$.
	Further, we estimate $\gamma$. Since the function $w_\alpha$ and exponent  are smooth, we expand in a Taylor series:
		\begin{equation*} \begin{split}
				&\gamma\leq (1- g(\eta)r)(\langle s_*, f^*\rangle r-\langle s_*, f_*\rangle r +w_\alpha(\xi,\eta))-w_\alpha(\xi,\eta)+h_1(r,\xi,\eta,\theta)r\\ &  =(\langle s_*, f^*\rangle-\langle s_*, f_*\rangle)r-  g(\eta) w_\alpha(\xi,\eta)r+h_1(r,\xi,\eta,\theta)r.
			\end{split}
		\end{equation*} 
  The residual terms of the exponent and  $w_\alpha$ have the form $ C(\xi,\eta,\theta)r^2/2$, where $C$ is the constant depending on $(\xi,\eta,\theta)$. We estimate the residuals terms by the function $h_1(r,\xi,\theta)r$, $\lim\limits_{r\rightarrow 0}h_1(r,\xi,\eta,\theta)=0$. Notice, that the function $h_1$ doesn't depend on $x(\cdot,\xi),y(\cdot,\eta)$.
Further in the paper, functions $h_i$ depend only on $r,\xi,\eta,\theta$ and
do not depend on the considered motion $x(\cdot)\in X(x_0, U_\alpha,\Delta)$ and $y(\cdot,\eta)$.

	From the definition  $p_0$, we obtain: $$\langle s_*, f^*\rangle+(1-u_\alpha(\xi))g(\xi)\leq H(\xi,s_*,u_\alpha(\xi))+h_2(r,\xi,\eta,\theta), \ h_2(r,\xi,\eta,\theta)\rightarrow0, \mbox{ as } r\rightarrow 0;$$
	From the definition $q_0$, we have:
	$$\langle s_*, f_*\rangle+(1-u_\alpha(\xi))g(\eta)\geq H(\eta,s_*,u_\alpha(\xi))-h_3(r,\xi,\theta), \ h_3(r,\xi,\eta,\theta)\rightarrow0, \mbox{ as } r\rightarrow 0;$$
	Let us substitute these values in the expression for $\gamma$. Therefore, we obtain
	\begin{equation*} \begin{split} &\gamma\leq H(\xi,s_*,u_\alpha(\xi))r+h_2(r,\xi,\eta,\theta)r-(1-u_\alpha(\xi))g(\xi)r-H(\eta,s_*,u_\alpha(\xi))r\\
	&+h_3(r,\xi,\eta,\theta)r+(1-u_\alpha(\xi))g(\eta)r.\end{split}
\end{equation*}
	From~\eqref{ineq}, we have  	\begin{equation*} \begin{split} & H(\xi,s_*,u_\alpha(\xi))-H(\eta,s_*,u_\alpha(\xi))-w_\alpha(\xi,\eta)g(\eta) \\
			& \leq H(\xi,s_*,u_\alpha(\xi))-H(\eta,s_*,u_\alpha(\xi))-w_\alpha(\xi,\eta)b\leq0,\end{split}
	\end{equation*}
where $b$ is introduced in assumption $A4$.
	Hence, \begin{equation*} \begin{split} & \gamma\leq h_2(r,\xi,\eta,\theta)r-(1-u_\alpha(\xi))g(\xi)r +h_3(r,\xi,\eta,\theta)r +(1-u_\alpha(\xi))g(\eta)r\\ & =(1-u_\alpha(\xi))(g(\eta)-g(\xi))r +h_{4}(r,\xi, \eta,\theta)r.\end{split}
	\end{equation*}
	
	Combining the previous inequality and  ~\eqref{ineq2}, we arrive at the estimate
	\begin{equation*} \begin{split} & 1-u_\alpha(\xi)\leq e^{-\int_0^r g(y(t))dt}(1-u_\alpha(\xi+f^*r))\\ & +(1-u_\alpha(\xi))(g(\eta)-g(\xi))r +h_{4}(r,\xi,\eta,\theta)r.\end{split}
	\end{equation*}
	Equivalently,
	\begin{equation*} \begin{split} & (1-u_\alpha(\xi))e^{\int_0^r g(y(t))dt}\leq (1-u_\alpha(\xi+f^*r))\\
			&+e^{\int_0^r g(y(t))dt}(1-u_\alpha(\xi))(g(\eta)-g(\xi))r +e^{\int_0^r g(y(t))dt}h_{4}(r,\xi,\eta,\theta)r.\end{split}
		\end{equation*}
	We add $(1-u_\alpha(\xi))e^{\int\limits_0^r g(x(t))dt}$ to the both parts of the previous inequality:
	\begin{equation*} \begin{split} & (1-u_\alpha(\xi))e^{\int_0^r g(x(t))dt}\leq (1-u_\alpha(\xi+f^*r))+
	(1-u_\alpha(\xi))\Bigl(e^{\int_0^r g(x(t))dt}-e^{\int_0^r g(y(t))dt}\Bigr)\\ &+
	e^{\int_0^r g(y(t))dt}(1-u_\alpha(\xi))(g(\eta)-g(\xi))r +e^{\int_0^r g(y(t))dt}h_{4}(r,\xi,\eta,\theta)r.\end{split}
\end{equation*}

	Let us expand in a Taylor series exponent and continue the values:
	\begin{equation*} \begin{split} &(1-u_\alpha(\xi))e^{\int_0^r g(x(t))dt}\leq (1-u_\alpha(\xi+f^*r)) +
	(1-u_\alpha(\xi))\Bigl(r g(\xi)-r g(\eta)\Bigr)\\
	&+\Bigl(1+r g(\eta)\Bigr)(1-u_\alpha(\xi))(g(\eta)-g(\xi))r+\Bigl(1+r g(\eta)\Bigr)h_{4}(r,\xi,\eta,\theta)r
	\\
	&\leq(1-u_\alpha(\xi+f^*r))+(1-u_\alpha(\xi))\Bigl( g(\xi)- g(\eta)\Bigr)r\\
 &+(g(\eta)-g(\xi))(1-u_\alpha(\xi))r +h_4(r,\xi,\eta,\theta)=(1-u_\alpha(\xi+f^*r))+h_4(r,\xi,\eta,\theta). \end{split}
\end{equation*}
	where $\lim\limits_{r\rightarrow 0}h_4(r,\xi,\eta,\theta)=0.$
	
To complete the proof, we will derive the conclusion of the Theorem 
 from~\eqref{main}. Further, we omit the argument $x(\cdot)$ in $\tau_\varepsilon$.
 We denote by $\mu(\cdot)$  the modulus of continuity of the function $\tilde{\sigma}$ and 
  set for $\alpha,\delta>0$
     $$l(\alpha,\delta)=-\ln\Bigl(1-\frac{\alpha+\mu(2\alpha) +h(\delta,x_0,\theta)\theta}{1- u(x_0)}\Bigr).$$ 
     
Now, we choose $\varepsilon>0$ and $\alpha>0$ satisfying the inequalities~\eqref{l1},\eqref{pars} and
$$ I+\ln(1- u(x_0))>l(\alpha,0).$$ This implies that there exists $\delta_0>0$ such that if
$\delta<\delta_0$, inequalities~\eqref{l1},\eqref{pars} and  \begin{equation} \label{l2}I+\ln(1- u(x_0))>l(\alpha,\delta)\end{equation} 
are valid. It follows from the properties  of the functions $\mu(\cdot)$ and $h(\cdot,x_0,\theta)$:
 $$ \lim\limits_{\alpha\rightarrow 0}\mu(2\alpha)=0, \ \lim\limits_{\delta\rightarrow 0}h(\delta,x_0,\theta)=0.$$ 
  Now, let a  partitions $\Delta$ be such that $\mbox{ diam} \Delta=\delta\leq \delta_0$. It produces the set  $X(x_0, U_\alpha, \Delta)$. Let $x(\cdot)\in X(x_0, U_\alpha, \Delta)$ and 
  \begin{equation} \label{ht}
      \hat{t}_1=\inf\{ t\in[0,\theta]: \mbox{ dist}(x(t), \mathcal{M}_1)\leq \varepsilon\};
\hat{t}_2=\inf\{ t\in[0,\theta]: \mbox{ dist}(x(t), \mathcal{M}_2)\leq \varepsilon\}.
  \end{equation} 
  
There are  the following  cases:
 \begin{enumerate}
	\item  $\hat{t}_1=\hat{t}_2=+\infty$;
	\item $\hat{t}_1<\hat{t}_2$;
	\item $\hat{t}_1>\hat{t}_2$.
\end{enumerate}
Case~1.   We consider inequality~\eqref{main} on $[0,\theta]$:
	\begin{equation*} \begin{split}  u_\alpha(x(\theta))&\leq 1-(1-u_\alpha(x_0))e^{\int_0^{\theta}g(x(t))dt}+ \theta e^{\int_0^{\theta}g(x(t))dt}h(\delta,x_0,\theta)\\
	&\leq 1-(1-u(x_0)-\alpha)e^{\int_0^{\theta}g(x(t))dt} + \theta e^{\int_0^{\theta}g(x(t))dt}h(\delta,x_0,\theta)\\
&	<1-(1-u(x_0))e^{b\theta} +(\alpha + \theta h(\delta,x_0,\theta))e^{\kappa\theta}<0.\end{split}
\end{equation*} Here, $\kappa$ is defined by~\eqref{kapa}. The latter is due to the choice $\alpha,\delta$ in~\eqref{pars}.	This contradicts with the condition $u_\alpha(x(\theta))>0$.
	Hence, the first case  is impossible.
	
	Case~2. Notice, that $\tau_\varepsilon= \hat{t}_1$.
 Let us consider inequality~\eqref{main} on the interval $[0,\tau_\varepsilon]$. 
	 
	  It follows from formula~\eqref{main} that
	$$1-u_\alpha(x_0)\leq e^{-\int_0^{\tau_\varepsilon} g(x(t))dt}(1-u_\alpha(x(\tau_\varepsilon)))+h(\delta,x_0,\theta)\tau_\varepsilon.$$ 
Let us remind that $u_\alpha(x_0)\leq u(x_0)+\alpha$.
From the inequalities $\|x(\tau_\varepsilon)-y_{\alpha}(x(\tau_\varepsilon))\|\leq2\alpha$ \cite[p.~246]{Subbotin}, $u(x)\geq \tilde{\sigma}(x)$ for all $x\in G\setminus M_2^\varepsilon$  and Proposition~2,  we deduce that
 \begin{equation*} u_\alpha(x(\tau_\varepsilon))>\tilde{\sigma}(y_{\alpha}(x(\tau_\varepsilon))) -\tilde{\sigma}(x(\tau_\varepsilon))+\tilde{\sigma}(x(\tau_\varepsilon))
 >-\mu(2\alpha)+\tilde{\sigma}(x(\tau_\varepsilon)).
\end{equation*}
	Plugging this into  inequality~\eqref{main}, we arrive at:
	\begin{equation} \label{ineq3} \Bigl(1- u(x_0)-\alpha-\mu(2\alpha)e^{-\int_0^{\tau_\varepsilon} g(x(t))dt}-h(\delta,x_0,\theta)\tau_\varepsilon\Bigr)e^{\int_0^{\tau_\varepsilon} g(x(t))dt}\leq 1-\tilde{\sigma}(x(\tau_\varepsilon)). \end{equation} Hence,
	\begin{equation*} \begin{split} &(1- u(x_0))\Bigl(1-\frac{\alpha+\mu(2\alpha)e^{-\int_0^{\tau_\varepsilon} g(x(t))dt}+h(\delta,x_0,\theta)\tau_\varepsilon}{1- u(x_0)}\Bigr)e^{\int_0^{\tau_\varepsilon} g(x(t))dt} \\
 &\leq 1-\tilde{\sigma}(x(\tau_\varepsilon)).
 \end{split}
\end{equation*}
	We take a logarithm from both sides of this inequality and  conclude:
	\begin{equation*} \begin{split} & \ln(1- u(x_0))+\int_0^{\tau_\varepsilon} g(x(t))dt \leq \ln (1-\tilde{\sigma}(x(\tau_\varepsilon)))  \\
	&-\ln\Bigl(1-\frac{\alpha+\mu(2\alpha)+h(\delta,x_0,\theta)\theta}{1- u(x_0)}\Bigr).\end{split}
\end{equation*}
	Notice, that $1>\tilde{\sigma}(x(\tau_\varepsilon))$, while $x(\tau_\varepsilon)$ belongs to the $\varepsilon$-neighbourhood  of the border~$\mathcal{M}_1$.   Let us recall that the Kruzhkov's transformation $$v(x)=-\ln(1-u(x)),$$ such that $u(\cdot)$ is the minimax solution of problem~\eqref{h2} iff $v(\cdot)$ is the minimax solution of problem~\eqref{h1}.
	Furthermore,
	\begin{equation}  \label{ineq1}\begin{split}  \ln(1- u(x_0))+\int_0^{\tau_\varepsilon} g(x(t))dt &\leq\ln(1-\tilde{\sigma}(x(\tau_\varepsilon))+l(\alpha,\delta),\\
	\hat{\sigma}(x(\tau_\varepsilon)) +\int_0^{\tau_\varepsilon} g(x(t))dt &\leq -\ln(1- u(x_0))+l(\alpha,\delta).\end{split}
\end{equation}

 Hence, we have the inequality:
$$\hat{\sigma}(x(\tau_\varepsilon)) +\int_0^{\tau_\varepsilon} g(x(t))dt\leq I.$$  

	Case~3. In this case, $\tau_\varepsilon=\hat{t}_2$. We again consider inequality~\eqref{main} on the interval $[0,\tau_\varepsilon]$. It follows from formula~\eqref{main} that
	$$1-u_\alpha(x_0)\leq e^{-\int_0^{\tau_\varepsilon} g(x(t))dt}(1-u_\alpha(x(\tau_\varepsilon)))+h(\delta,x_0,\theta)\tau_\varepsilon.$$ Using inequality~\eqref{ineq3}, we obtain the 
\begin{equation*} \begin{split} &
    (1- u(x_0))e^{\int_0^{\tau_\varepsilon} g(x(t))dt} \leq \\
    &1-\tilde{\sigma}(x(\tau_\varepsilon))+ \alpha  e^{\int_0^{\tau_\varepsilon} g(x(t))dt} +\mu(2\alpha)+h(\delta,x_0,\theta)\tau_\varepsilon e^{\int_0^{\tau_\varepsilon} g(x(t))dt}.
    \end{split}
\end{equation*} 

Notice, that $1-\tilde{\sigma}(x(\tau_\varepsilon))=0$,  while $x(\tau_\varepsilon)$ belongs to the $\varepsilon$-neighborhood of the border $\mathcal{M}_2$. From that, we obtain
$$(1- u(x_0)) \leq \alpha+\mu(2\alpha)e^{-\int_0^{\tau_\varepsilon} g(x(t))dt} +h(\delta,x_0,\theta)\theta\leq  \alpha+\mu(2\alpha)+h(\delta,x_0,\theta)\theta.$$ 

Since $ \lim\limits_{\alpha\rightarrow 0}\mu(2\alpha)=0, $ $ \lim\limits_{\delta\rightarrow 0}h(\delta,x_0,\theta)=0$, the following inequality holds: $$ (1- u(x_0))> \alpha+\mu(2\alpha)+h(\delta,x_0,\theta)\theta. $$  It contradicts~\eqref{main}.
Hence, case~3 is impossible.

Combining the aforementioned  cases, we conclude that 
$$ J_1^\varepsilon(x_0,U_\alpha,\Delta)\leq I.$$
	 
\end{proof}

Now, we construct $\varepsilon$-optimal strategies for the second player.

Let  $x_0\in G$ be given and $u(\cdot)$ is a minimax solution of problem~\eqref{h2}. We choose $I>0 $ and
$I<-\ln (1-u(x_0))$. Then, according to the definition of the minimax solution,  there exists a subsolution $u_\natural$ of problem~\eqref{h2} such that $$I<-\ln(1-u_\natural(x_0)).$$
 We can construct $\varepsilon$-optimality strategy for the second player in the similar way.

We transform the subsolution  $u_\natural(\cdot)$ in the following way
\begin{equation}\label{v-a}
	v_\alpha(x)=\max_{y\in \mathbb{R}^d}[ u_\natural(y)-w_\alpha(x,y)],
\end{equation}
where $w_\alpha$ satisfies~\eqref{w}.

The function $y\rightarrow u_\natural(y)-w_\alpha(x,y)$ is upper semicontinuous, therefore maximum in expression~\eqref{v-a} is attained at a point $y_\alpha$  such that $\|x-y_\alpha\|\leq1$.  It is proved in ~\cite{Subbotin} that $\|x-y_\alpha\|\leq 2\alpha$, where $y_\alpha$ is the argmax in~\eqref{v-a}.
We set  $s_\alpha(x)=-D_xw_\alpha(x,y_\alpha(x))=D_yw_\alpha(x,y_\alpha(x))$.

\begin{theorem} Let a point $x_0 \in G$ be given.  Then for every $I<-\ln(1-u_\natural(x_0))$, there exist  $\varepsilon>0$,
	 $\alpha>0$ and $\delta_0>0$ such that for every partition $\Delta$ satisfying $\mbox{diam }\Delta <\delta_0$, one has 
 the estimate:
	\begin{equation} \label{tmain2}   \quad
	J_2^\varepsilon(x_0, V_\alpha, \Delta) \geq I,\end{equation}
	where  $V_\alpha$ is a feedback strategy of the second player, defined by $V_\alpha(x)=q_0(x,s_\alpha(x))$. Here, $q_0$ satisfies formula~\eqref{predstr}.
\end{theorem}
\begin{proof}
We follow the scheme of the proof of  Theorem~4. We consider two cases of the initial data: 
\begin{itemize}
    \item $u_\natural(x_0)<1$;
    \item $u_\natural(x_0)=1$.
\end{itemize}

	Further, let $\theta<\frac{I}{b}$. Let us remind the distance $d(M_1,M_2)$ between the sets $M_1$ and $M_2$:
 $$d(M_1,M_2)=\min\Bigl\{ \| x-y\|:  \ x\in M_1, y\in M_2\Bigr\} .$$
 
We choose $0<\varepsilon_0<d(M_1,M_2)$ such that $u_\natural(\cdot)$ is continuous on the set $M=M_1^{\varepsilon_0}\cap K$ where $K$ is defined by~\eqref{setk}.  Thus, the modulus of continuity on the set $M$
 $$\omega_{M}(\varepsilon)=\sup\{|u_\natural(x_1)-u_\natural(x_2)|:  \ x_1,x_2\in M, \  \|x_1-x_2\|\leq \varepsilon\}$$ is well defined.
From the upper semicontinuity of the function $u_\natural$ and the estimate $\|y_\alpha(x)-x\|\leq 2\alpha$, we deduce that, for every bounded set $N\subset G $, there exists a number $\mu(\alpha)$ such that $$\forall x\in N, \ u_\natural(y_\alpha(x))\leq u_\natural(x)+\mu(\alpha), \ \lim\limits_{\alpha\rightarrow 0}\mu(\alpha)=0.$$
Now, we consider the first case when $u_\natural(x_0)<1$. Let us denote $$l(\varepsilon,\alpha,\delta)=\ln\Bigl(1+\frac{\alpha+h(\delta,x_0,\theta)\theta+\mu(\alpha)+\omega_{M^{\varepsilon_0}}(\varepsilon)+\lambda_{\tilde{\sigma}}\varepsilon}{1-u_\natural(x_0)}\Bigr),$$
 where $\lambda_{\tilde{\sigma}}$ is Lipschitz constant for the function $\tilde{\sigma}$.  Notice, that $\lim\limits_{\varepsilon,\alpha,\delta \rightarrow 0}l(\varepsilon,\alpha,\delta)=0. $

 Now, we choose $\varepsilon\in(0,\varepsilon_0]$ and $\alpha>0$ satisfying  inequality~\eqref{l1} and
$$ -I-\ln(1- u(x_0))>l(\varepsilon,\alpha,0).$$ This implies that there exists $\delta_0>0$ such that for  
$\delta\in(0,\delta_0)$ the following inequality holds: \begin{equation}  \label{l}
         l(\varepsilon,\alpha,\delta)<-\ln(1-u_\natural(x_0))-I.
    \end{equation} 
   Let $\Delta$ be a partition such that $\mbox{ diam} \Delta=\delta<\delta_0$. It produces the set of motions  $X(x_0, V_\alpha, \Delta)$.
    Further, we consider a step-by-step motion $x(\cdot) \in X(x_0,V_\alpha, \Delta)$. First,  assume that $t_i\in \Delta$, $t_i<\theta$ and $\mbox{dist}(x(t_i);\partial G)>\varepsilon$. Then, for all  $r\in [t_i,t_{i+1}]\cap [0,\theta]$,
	\begin{equation} \label{main2}
		1-v_\alpha(x(t_i))\geq e^{-\int_{t_i}^r g(x(t))dt}(1-v_\alpha(x(r)))-h(\delta,x(t_i),\theta)(r-t_i),
	\end{equation}
	where $\lim\limits_{\delta\rightarrow 0} h(\delta,x(t_i),\theta)=0$. Further, we omit the argument of $\tau_\varepsilon$. The proof of formula ~\eqref{main2} is  the same as the proof of~\eqref{main}. 	Now, we deduce formula~\eqref{tmain2}  from formula~\eqref{main2}.
		Let $x(\cdot) \in X(x_0,V_\alpha, \Delta)$. 
Below, we will consider the following  cases for $\hat{t}_1$, $\hat{t}_2$ defined by \eqref{ht}:
 \begin{enumerate}
	\item  $\hat{t}_1=\hat{t}_2=+\infty$;
	\item $\hat{t}_1>\hat{t}_2$;
	\item $\hat{t}_1<\hat{t}_2$.
\end{enumerate}

	Case~1. The second player evades the $\varepsilon$-neighbourhood of sets $\mathcal{M}_1$ and $\mathcal{M}_2$ before the instant $\theta$ using the strategy $V_\alpha$.  Since $\tau_\varepsilon>\theta$, we obtain 
 $$ \int\limits_0^{\tau_\varepsilon} g(x(t))dt>\int\limits_0^\theta bdt\geq I,$$ due to the choice of $\theta$.  Let us recall that $\hat{\sigma}(x(\tau_\varepsilon))\geq0$, hence $
 \hat{\sigma}(x(\tau_\varepsilon))+\int\limits_0^{\tau_\varepsilon} g(x(t))dt\geq I$.

Case~2. The second player touches the $\varepsilon$-neighborhood of the set $M_2$ and avoids the set $M_1^\varepsilon$, so, $\tau_\varepsilon=\hat{t}_2$ and
 $$\hat{\sigma}(x(\tau_\varepsilon))+\int\limits_0^{\tau_\varepsilon} g(x(t))dt=+\infty.$$ 

	Case~3. In this case $\tau_\varepsilon=\hat{t}_1$. 
  From~\eqref{v-a}, we obtain
 \begin{equation} \label{eq3-1}
  v_\alpha(x_0)\geq u_\natural(x_0)-\alpha   
 \end{equation} 
 and $ v_\alpha(x(\tau_\varepsilon))= u_\natural(y_\alpha(x(\tau_\varepsilon)))-w_\alpha(y_\alpha(x(\tau_\varepsilon)),x(\tau_\varepsilon))$. Thus, 
 we have 
 \begin{equation} \label{eq3-2}
    v_\alpha(x(\tau_\varepsilon)) <u_\natural(x(\tau_\varepsilon)) +\mu(\alpha).  
 \end{equation}
  We substitute  inequalities~\eqref{eq3-1},~\eqref{eq3-2} in~\eqref{main2} on the interval $[0, \tau_\varepsilon]$:
 $$ \Bigl(1-u_\natural(x_0)+\alpha+h(\delta,x_0,\theta)\tau_\varepsilon+\mu(\alpha)e^{-\int_{0}^{\tau_\varepsilon} g(x(t))dt}\Bigr)e^{\int_{0}^{\tau_\varepsilon} g(x(t))dt}\geq 1-u_\natural(x(\tau_\varepsilon)).$$ 
 
 Now, we find a point $x^*\in \mathcal{M}_1$ and $\|x(\tau_\varepsilon)-x^*\|=\mbox{dist } (x(\tau_\varepsilon);\mathcal{M}_1)$. Notice,  that
 $\|x(\tau_\varepsilon)-x^*\|= \varepsilon$.
 
Let us estimate $$|u_\natural(x(\tau_\varepsilon))-\tilde{\sigma}(x^*)|\leq\omega_{M_1^{\varepsilon_0}}(\varepsilon).$$
 Moreover, $$\|\tilde{\sigma}(x^*)-\tilde{\sigma}(x(\tau_\varepsilon))\|\leq\lambda_{\tilde{\sigma}}\|x^*-x(\tau_\varepsilon)\|=\lambda_{\tilde{\sigma}} \varepsilon. $$  Hence, we get
$$1-u_\natural(x(\tau_\varepsilon))+\tilde{\sigma}(x^*)+(\tilde{\sigma}(x(\tau_\varepsilon))-\tilde{\sigma}(x^*)) -\tilde{\sigma}(x(\tau_\varepsilon))$$ 
$$>1-\omega_{M^{\varepsilon_0}}(\varepsilon)-\lambda_{\tilde{\sigma}}\varepsilon-
\tilde{\sigma}(x(\tau_\varepsilon)). $$
   Applying these estimates to formula~\eqref{main2}, we obtained
   $$ \Bigl(1-u_\natural(x_0)+\alpha+h(\delta,x_0,\tau_\varepsilon)\tau_\varepsilon+[\mu(\alpha)+\omega_{M^{\varepsilon_0}}(\varepsilon)+\lambda_{\tilde{\sigma}}\varepsilon]e^{-\int_{0}^{\tau_\varepsilon} g(x(t))dt}\Bigr)e^{\int_{0}^{\tau_\varepsilon} g(x(t))dt}$$
   $$\geq 1-\tilde{\sigma}(x(\tau_\varepsilon)).$$

Since $\tau_\varepsilon<\theta$, we have
$$\Bigl(1-u_\natural(x_0)+\alpha+h(\delta,x_0,\tau_\varepsilon)\theta+\mu(\alpha)+\omega_{M^{\varepsilon_0}}(\varepsilon)+\lambda_{\tilde{\sigma}}\varepsilon \Bigr)e^{\int_{0}^{\tau_\varepsilon} g(x(t))dt}$$
   $$\geq 1-\tilde{\sigma}(x(\tau_\varepsilon)).$$
We take logarithm from the both parts of this inequality, and derive the estimate
    $$ \ln(1-u_\natural(x_0))+\int_{0}^{\tau_\varepsilon} g(x(t))dt$$ $$+\ln\Bigl(1+\frac{\alpha+h(\delta,x_0,\theta)\theta+\mu(\alpha)+\omega_{M^{\varepsilon_0}}(\varepsilon)+\lambda_{\tilde{\sigma}}\varepsilon}{1-u_\natural(x_0)}\Bigr)\geq \ln(1-\tilde{\sigma}(x(\tau_\varepsilon))).$$ 
    
Hence, $$-\ln(1-\tilde{\sigma}(x(\tau_\varepsilon)))+\int_{0}^{\tau_\varepsilon} g(x(t))dt\geq I.  $$
Thus, we deduce that $\forall \varepsilon<\varepsilon_0,$ $ \alpha, \delta $ satisfying~\eqref{l} the following inequality
$$J_2^\varepsilon(x_0,V_\alpha,\Delta_n)\geq I$$
is valid.

Notice, that if $\hat{\sigma}(\cdot)$ is a constant function and $g(\cdot)$ is a constant function  (that is we deal with time optimal problem), then  case~3 is impossible. It is proved in \cite[p.767]{Kumkov}.

Let us consider the second case when $u_\natural(x_0)=1$.  Notice, that $\theta<+\infty$.
We define a function
\begin{equation*} 
    l_1(\varepsilon,\alpha,\delta)=(\alpha+h(\delta,x_0,\theta)\theta+\mu(\alpha)+\omega_{M^{\varepsilon_0}}(\varepsilon)+\lambda_{\tilde{\sigma}}\varepsilon)e^{\kappa\theta}).
\end{equation*}
Now, we choose $\varepsilon\in (0, \varepsilon_0]$ and $\alpha>0$  satisfying inequality~\eqref{pars} and
$$l_1(\varepsilon,\alpha,0)   < e^{-\Sigma} .$$ Here, $\Sigma$ is a constant from condition $A5$.
This implies that there exists $\delta_0>0$  such that for $\delta\in (0,\delta_0)$ the following
inequality holds
\begin{equation} \label{l11}
    l_1(\varepsilon,\alpha,\delta)   < 1-\tilde{\sigma}(x(\tau_\varepsilon)).
          \end{equation}
 Furthermore, let $\Delta$ be a partition such that $\mbox{diam}\Delta=\delta<\delta_0 $. It produces the set of motions
$X(x_0, V_\alpha, \Delta).$  
Recall, that $\varepsilon,\alpha,\delta$ satisfying~\eqref{pars},\eqref{l11} and $\varepsilon<\varepsilon_0$.
              For $x(\cdot)\in X(x_0, V_\alpha, \Delta)$,   we consider the same steps as above for~\eqref{ht}. 
     
If $\hat{t}_1=\hat{t}_2=+\infty$,  the proof coincidences with the same case  when $u_\natural(x_0)<1$. 

If $\hat{t}_1>\hat{t}_2$, the proof coincidences with the same case $u_\natural(x_0)<1$.

If $\hat{t}_1<\hat{t}_2$,
from formula~\eqref{main2}, we receive
\begin{equation*}
       l_1(\varepsilon,\alpha,\delta)
   \geq 1-\tilde{\sigma}(x(\tau_\varepsilon)). 
     \end{equation*}
 Thus, for $\varepsilon,\alpha,\delta$ satisfying~\eqref{pars},\eqref{l11} and $\varepsilon<\varepsilon_0$, we  come to  contradiction with formula~\eqref{main2}. So, this case is impossible under $u_\natural(x_0)=1$. 

 Combining the aforementioned cases, we conclude that
$$J_2^\varepsilon(x_0,V_\alpha,\Delta_n)\geq I$$
 holds true.
	
 \end{proof}

\begin{remark}
Under assumptions $A1$--$A4$ in problem~\eqref{Kolpakova:eq2} --~\eqref{funcl} there exists the value function $\operatorname{Val}$ and
$u(x)=1-e^{-\operatorname{Val}(x)}$, where $u$ is the unique minimax solution of problem~\eqref{h2}.
\end{remark}

\section{Conclusion}
This paper is generalized known results about the existence of the value function in the time-optimal problems and in the time-optimal problems with lifeline. Besides, we proved the existence of the minimax solution in the corresponding Dirichlet problem for the Hamilton---Jacobi equation under standard condition on the dynamics and the boundary function. We do not assume additional conditions on the boundary of the target set and the lifeline. We shew the coincidence the value function and Kruzhkov's transform of the minimax solution. Notice, that under our assumptions the value function can be discontinuous.

In the paper, we consider the payoff with integral part depending on only the trajectory of the players. We plan to consider the case  when the payoff will depend on not only the trajectory  but on the controls of the players. We will research the existence of the value function and coincidence with the minimax solution of Hamilton---Jacobi equation in this case. 

\bmhead{Acknowledgments}
The author would like to thank Prof. Yu. Averboukh for his valuable comments and suggestions. 


\bibliography{sn-6}
\end{document}